\documentclass[12 pt]{article}
\usepackage{amsfonts,amsmath,color,amsthm,amssymb,verbatim}

\title{Permutation Patterns in Latin Squares}
\author
{Michael J. Earnest\\
Department of Mathematics\\
University of Southern California
\and
Samuel C. Gutekunst\\
Department of Mathematics\\
Harvey Mudd College}

\begin{document}
\def\ep{\varepsilon}
\def\aln{\alpha n}
\def\maln{(2-\alpha)n}
\def\lr{\left(}
\def\lc{\left\{}
\def\rc{\right\}}
\def\rr{\right)}
\def\nt{{n\choose2}}
\def\nc{\lceil n/2\rceil}
\def\ic{\lceil i/2\rceil}
\def\nr{{n\choose r}}
\def\ns{{n\choose s}}
\def\tv{{d_{{\rm TV}}}}
\def\P{{\rm Po}}
\def\cl{{\cal L}}
\def\ta{{\tilde{a}}}
\def\tb{{\tilde{b}}}
\def\qed{\vbox{\hrule\hbox{\vrule\kern3pt\vbox{\kern6pt}\kern3pt\vrule}\hrule}}
\def\ca{{\cal A}}
\def\vca{{\vert\cal A\vert}}
\def\ck{{\cal K}}
\def\ch{{\cal H}}
\def\cc{{\cal C}}
\def\cs{{\cal S}}
\def\p{\mathbb P}
\def\v{\mathbb V}
\def\e{\mathbb E}
\def\z{\mathbb Z}
\def\l{\lambda}
\def\a{\alpha}
\def\exp{\mathrm{exp}}
\def\n{\noindent}
\newtheorem{thm}{Theorem}
\newtheorem{lm}[thm]{Lemma}
\newtheorem{rem}[thm]{Remark}
\newtheorem{cor}[thm]{Corollary}
\newtheorem{exam}[thm]{Example}
\newtheorem{prop}[thm]{Proposition}
\newtheorem{defn}[thm]{Definition}
\newtheorem{cm}[thm]{Claim}
\newtheorem{conj}[thm]{Conjecture}
\maketitle

\begin{abstract}

In this paper we study pattern avoidance in Latin Squares, giving us a two dimensional analogue of the well-studied notion of pattern avoidance in permutations.  Our main results include enumerating and characterizing 
 the Latin Squares which avoid patterns of length three and a generalization of the Erd\H os-Szekeres theorem.  We also discuss equivalence classes among longer patterns, and conclude by describing open questions of interest both in light of pattern avoidance and their potential to reveal information about the structure of Latin Squares.  Along the way, we show that classical results need not generalize trivially, and we demonstrate techniques that may help answer future questions.

\end{abstract}
\section{Introduction}

A permutation of length $n$ is a rearrangement  of the numbers $\{1, 2, \ldots, n\}$ and many interesting questions have been asked and answered about the structure of permutation classes.  In particular, pattern containment and avoidance, which we will formally define shortly, ask about the types of subsequences a permutation does and does not have.  

A natural generalization of a permutation is a \emph{Latin Square}, which we will also introduce below.  Each row and column of a Latin Square is just a permutation, and so questions about patterns in permutations readily generalize to questions about patterns in Latin Squares.  Latin Squares are exciting objects to study by themselves, and we will begin this paper by formally defining pattern avoidance in Latin Squares. Little work seems to have been done in this area.

Sections $3$ through $5$ extend classical results from pattern avoidance in permutations.  We begin by enumerating and characterizing Latin squares that avoid patterns of length three.  In Section $4$ we discuss avoidance of longer patterns, and in Section $5$, discuss pattern containment.

Our final section, in our minds, is one of the most critical parts of this paper: a discussion of several open questions.  Answers to these questions will not only be exciting additions to the current results in the field of pattern avoidance, but they may also lay the foundation for answering several important questions about the structure and number of Latin Squares.  Most of the contents of this paper were presented at the $11^{\text{th}}$ Annual Permutation Patterns conference, and after the talk, many participants came up with additional questions.  We have added these to our list in Section $6$.

\section{Background}

We previously described a Latin Square as a set of permutations.  More concretely, an $n^{th}$ order Latin Square is an $n$ by $n$ grid in which the numbers $1, 2, ..., n$ (often called symbols) are each used exactly once in each row and column.  We readily know that there are $n!$ permutations of $n$ and so it comes as a surprise that the number of Latin Squares of order $n$ is only known up to $n=11$ \cite{McKay}. If we let $L_n$ be the number of $n^\text{th}$ order Latin Squares, then the best known bounds for $L_n$ are very far apart. For example, van Lint and Wilson \cite[p.~187]{VLW} give upper and lower bounds which differ asymptotically by a factor of $n^n$.

We can naturally extend the definition of pattern avoidance in permutations to pattern avoidance in Latin Squares.  This definition first requires an understanding of pattern containment in permutations: a permutation of size $n$ is said to \emph{ contain} a pattern (also a permutation) of size $k\leq n$ if a subsequence of the permutation is order isomorphic to the pattern.  If a permutation does not  contain a pattern, it is said to \emph{avoid} it.  For example, the permutation $13254$ contains $123$ because the subsequence $135$ is in the same relative order (that is, strictly increasing) as the pattern $123$.  Conversely, $13254$ avoids $321$ because it does not contain a strictly decreasing subsequence of length three.

To extend the preceding definition to the setting of Latin Squares, note that each row and each column of a Latin Square can be viewed as a permutation by reading the rows and the columns from left to right and top to bottom, respectively.  We  define a Latin Square's \emph{row permutations} to be the $n$ permutations corresponding to the rows in this manner and we define the \emph{column permutations} similarly.  Then:

\begin{defn}
A Latin Square avoids a pattern $\pi$ if all row and column permutations avoid $\pi$. The number of $n^\text{th}$ order Latin Squares avoiding $\pi$ will be denoted $L_n(\pi)$.
\footnote{While this definition is particularly natural, we note that there are other ways of defining pattern avoidance in Latin Squares. For example, each symbol $k$ of a Latin Square determines a permutation $\pi$, where $\pi(i)=j$ when the entry in row $i$, column $j$ is $k$. One could also require these symbol permutations to avoid a pattern in order to say that the full square does. This convention could make sense, and there is a way to view and define Latin Squares so that the distinction between rows, columns and symbols is arbitrary  (see Chapter 17 in \cite{VLW}). Furthermore, since Latin Squares are two dimensional, it might also make sense to study Latin Squares avoiding a ``two dimensional pattern.''  In this paper will study pattern avoidance using Definition 1, though we will discuss  these alternative definitions of pattern avoidance in Section 6.}
\end{defn}

The canonical question of pattern avoidance  has been, given some pattern $\pi$, how many permutations avoid (or equivalently, contain) $\pi$.  One of the earliest results was that the number of permutations of length $n$ avoiding any pattern of length three (e.g. any permutation of $\{1, 2, 3\}$) is just $\frac1{n+1}\binom{2n}{n}$, the $n^\text{th}$ Catalan number.  This result is proved in chapter 4 of \cite{Bona}.

One might suspect that if $\pi$ and $\pi'$ are patterns of the same length, then the same number of permutations will avoid them; however, this is not true in general. When, for any $n$, the number of permutations in $S_n$ avoiding $\pi$ and $\pi'$ are the same,  these patterns are said to be \emph{Wilf-equivalent}.  In addition to enumerating permutations avoiding a pattern, characterizing these equivalence classes is a central question in the study of permutation patterns.

\section{Avoiding Patterns of Length Three}  
One of the first results in classical pattern avoidance was the enumeration of permutations which avoid patterns of length three.  In this section we ask the same question for Latin Squares and will find the following result:

\begin{thm}
\label{thm:avoid}
For any $\pi\in S_3$, $L_n(\pi)=n.$
\end{thm}

To prove this result, we begin by considering a less restrictive case: the number of Latin Squares avoiding a pattern in just the columns.  We can count these Latin Squares using the following proposition:

\begin{prop}
For any permutation $\sigma\in S_n$, there is exactly one Latin Square avoiding the pattern $123$ in the columns with $\sigma$ as its first row.
\end{prop}
\begin{proof} 
Suppose that the top row has been fixed as $\sigma$ and consider the column beginning with a $1$. The rest of the entries must be in decreasing order: if there were any two in increasing order, they would form a 123 pattern with the 1 in the top row. Thus, the only possible permutation for this column is \\ $1,n,(n-1),\ldots,3,2$.

The column whose first entry is $2$ must similarly be completed as \\ $2, 1, n,(n-1),\ldots, 3$ in a decreasing order.  This claim follows because the numbers $3$ through  $n$ must be in strictly decreasing order, and to avoid conflict with our $1^{st}$ column, $n$ cannot be in the second row. 

% They pointed out the last column is a special case, so I changed the below paragraph to be only for j<n.
% The j=n is handled at the end, which is easy since there is only one empty column left

Now proceed iteratively.  To fill out the column beginning with $j$, for $j<n$,  all elements greater than $j$ must be in a decreasing order, and $n$ cannot be placed in the first $j$ rows.  The elements greater than $j$ are then forced to be placed in the bottom $n-j$ rows.  To complete the column, the remaining numbers $2,\dots, j-1$ must be in decreasing order to avoid forming a $123$ pattern with $n$. Figure \ref{proof} shows an illustration of this process when $n=4$. \begin{figure}[h]
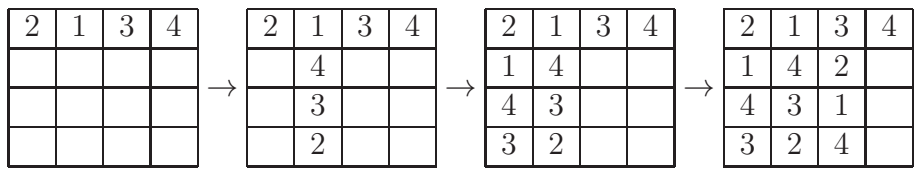
 
$$
\begin{tabular}{|c|c|c|c|}\hline
2&1&3&4\\\hline
&&&\\\hline
&&&\\\hline
&&&\\\hline
\end{tabular}
\to
\begin{tabular}{|c|c|c|c|}
\hline
2&1&3&4\\\hline
&4&&\\\hline
&3&&\\\hline
&2&&\\\hline
\end{tabular}
\to
\begin{tabular}{|c|c|c|c|}
\hline
2&1&3&4\\\hline
1&4&&\\\hline
4&3&&\\\hline
3&2&&\\\hline
\end{tabular}
\to
\begin{tabular}{|c|c|c|c|}
\hline
2&1&3&4\\\hline
1&4&2&\\\hline
4&3&1&\\\hline
3&2&4&\\\hline
\end{tabular}$$
\caption{Proof method of Proposition 2.}\label{proof}
\end{figure}

This leaves one unfilled column, which can only be completed one way to avoid repeats in the rows. This method will construct a unique Latin Square avoiding $123$ in the columns with $\sigma$ as its top row. \end{proof} 

Our proof readily generalizes for any of the other five permutations of length three.  To avoid $132$, first consider the $1$ in the top row and place the remaining elements in an increasing order.  To avoid $312$ or $321$, act similarly but first consider the $n$ in the first row.  To avoid $231$ and $213$, instead begin with the bottom row.

Because one row effectively determines a unique Latin Squares avoiding a pattern of length three in the columns, we obtain the following corollary. 

\begin{cor}
The number of $n^{th}$ order Latin Squares avoiding a pattern of length three in just the columns (or rows) is n!
\end{cor}

The above work also reveals a very interesting structure for pattern avoidance in the columns:
\begin{rem}
In a Latin Square avoiding $123$, $231$ or $312$, each entry is one less than the one above it (mod $n$). Thus, all columns are of the form $i, i-1, \ldots, 1, n, \ldots, i+1.$  When avoiding $132$, $213$, or $321$, all columns are instead increasing and of the form $i, i+1, \ldots, n, 1, \ldots, i-1.$ 

\end{rem}
 Note that these results pertain, respectively, to the even and odd permutations of $S_3$.   As above, this remark applies similarly to the rows.

We are now ready for the proof of Theorem 2.

\begin{proof} 

Let $\pi$ be a permutation in  $S_3$. To construct a Latin Square avoiding $\pi$, we have $n$ choices for which number is placed in the top left box of our Latin Square.  Using the above remark, there is exactly one way to complete this row 
(as $i, i-1, \ldots, 1, n, \ldots, i+1$ if $\pi$ is even, or $i, i+1, \ldots, n, 1, \ldots, i-1$, if $\pi$ is odd).  We now have one number in each column, which from Proposition 3 shows there is only one way to complete each column.

Since each of the $n$ choices for where this first element can go produces exactly one Latin Square avoiding $\pi$, this completes our proof.
\end{proof}

\begin{rem}
The $n$ $n^{\text{th}}$ order Latin Squares which avoid any particular pattern of length three have a structure that is worth noting. The Latin Squares which avoid $123, 231$, and $312$ are of the form shown in Figure \ref{123}, and the Latin Squares which avoid $132, 213$, and $321$ are of the form shown in Figure \ref{321}.

\begin{figure}[h]
\[
\begin{array}{|c|c|c|c|c|c|c|c|c|c|}
\hline
i&i-1&&&1&n&&&i+2&i+1\\\hline
i-1&\ddots&&1&n&&&i+2&i+1&i\\\hline
&&1&n&&&i+2&i+1&i&\\\hline
&1&n&\ddots&&i+2&i+1&i&&\\\hline
1&n&&&i+2&i+1&i&&&\\\hline
n&&&i+2&i+1&i&&&&1\\\hline
&&i+2&i+1&i&&\ddots&&1&n\\\hline
&i+2&i+1&i&&&&1&n&\\\hline
i+2&i+1&i&&&&1&n&\ddots&\\\hline
i+1&i&&&&1&n&&&i+2\\\hline
\end{array}
\]
\caption{ The general form of a 123, 231, or 312 avoiding Latin Square.}\label{123}
\end{figure}

\begin{figure}[h!]
\[
\begin{array}{|c|c|c|c|c|c|c|c|c|c|}
\hline
i&i+1&&&n&1&&&i-2&i-1\\\hline
i+1&\ddots&&n&1&&&i-2&i-1&i\\\hline
&&n&1&&&i-2&i-1&i&\\\hline
&n&1&\ddots&&i-2&i-1&i&&\\\hline
n&1&&&i-2&i-1&i&&&\\\hline
1&&&i-2&i-1&i&&&&n\\\hline
&&i-2&i-1&i&&\ddots&&n&1\\\hline
&i-2&i-1&i&&&&n&1&\\\hline
i-2&i-1&i&&&&n&1&\ddots&\\\hline
i-1&i&&&&n&1&&&i-2\\\hline
\end{array}
\]
\caption{ The general form of a 132, 213, or 321 avoiding Latin Square.}\label{321}
\end{figure}
\end{rem}

Every row and column  is in a cyclic increasing or decreasing structure where adjacent elements differ by one (mod $n$).  In addition, we earlier saw that there were $n!$ Latin Squares avoiding a pattern of length three in just the columns, and $n!$ avoiding it in just the rows.  When we force both restrictions we find that only $n$ Latin Squares satisfy both avoidance criteria.

By examining the above Latin Squares, we can also see the following corollary:

\begin{cor}
A Latin Square contains either all of the patterns $\{123,231,312\}$, or none of them. This also holds for $\{132,213,321\}$. 
\end{cor}

Given the previous result, the proof of this corollary is straightforward. If a Latin Square does not contain any of $\{123,231,312\}$, it must be in the decreasing form shown above and it will not contain the others. However, out of context, it is somewhat surprising that any Latin Square with three terms in a row or column in increasing order must have three terms in the relative order 231 and 312.

 We have now seen that all patterns of length three are also Wilf-equivalent for Latin Squares and that the growth rate of the number of these Latin Squares is polynomial as opposed to exponential (as is the case for permutations).

\section{Avoidance of Larger Patterns}

Computing $L_n(\pi)$ for a general pattern $\pi$ of length greater than three is considerably more difficult. As of yet, we know of no simple algorithm for filling in a partially completed Latin Square so that it will avoid a permutation in $S_4$. We begin this section with a much more tractable question: counting $L_n(\pi)$, for $\pi\in S_n$, in terms of the total number of Latin Squares.

\begin{thm}
For any $\pi\in S_n$, $L_n(\pi)=\lr\frac{n!-n}{n!}\rr^2L_n.$
\end{thm}

\begin{proof}
Let $\pi,\rho$ be permutations in $S_n$. Given a $\pi$-avoiding  $n^{th}$ order Latin Square, apply the permutation $\rho\circ\pi^{-1}$ to each entry. Doing so will create a bijection from Latin Squares which avoid $\pi$ to those which avoid $\rho$, so that $L_n(\pi)$ for $\pi\in S_n$ only depends on $n$. We first count Latin Squares avoiding any $\pi\in S_n$ in the columns. Let the number of these be  $\ell_n(\pi)$. 

Let two Latin Squares be r-equivalent if they are related by a permutation of rows. Each r-equivalence class will be of size $n!$. Let the $i^\text{th}$ column of a Latin Square, $S$, be $\sigma_i$. The $n$ permutations $\pi\circ\sigma_i^{-1}$, for $1\le i\le n$, are the only ones which, when applied to the rows of $S$, cause the result to contain $\pi$ in a column. Thus, each r-equivalence class will contain $n!-n$ Latin Squares which avoid $\pi$ in the columns, so that $\ell_n(\pi)=\frac{n!-n}{n!}\cdot L_n$. By similar logic, if we partition Latin Squares which column-avoid $\pi$ into c-equivalence classes up to permutation of columns, each c-equivalence class of size $n!$ will contain $n!-n$ Latin Squares which avoid $\pi$ in the rows and columns. Thus, we have $$L_n(\pi)=\frac{n!-n}{n!}\cdot\ell_n(\pi)=\lr\frac{n!-n}{n!}\rr^2L_n.$$  \end{proof}

As noted earlier, it is not generally true that $L_n(\pi)$ = $L_n (\pi')$ when $\pi$ and $\pi'$ are patterns of the same length. We say that $\pi$ and $\pi'$ are Wilf-equivalent in Latin Squares when $L_n(\pi)=L_n(\pi')$ for all $n$.

In classical pattern avoidance, Wilf-equivalence classes have rich structure. Let the $complement$ of a permutation, $\pi^c$, be given by $\pi^c(i)=(n+1)-\pi(i)$, and the reverse, $\pi^\text{rev}$, by $\pi^\text{rev}(i)=\pi(n+1-i)$. It is easy to show $\pi$ is Wilf-equivalent to its complement and reverse in permutations, and a quick proof shows $\pi$ is similarly equivalent to its $inverse$ (which satisfies $\pi^{-1}(\pi(i))=i$). 

There are many other Wilf-equivalences in permutations; for example, $S_n(4132)=S_n(3142)$, as shown in \cite{Stankova} (where $S_n(\pi)$ denotes the number of permutations of length $n$ avoiding a pattern $\pi$). Nontrivial equivalences exist for arbitrarily large patterns. For $\pi_1\in S_n$ and $\pi_2\in S_m$, let $\pi_1\oplus \pi_2$ be the permutation in $S_{n+m}$, given by applying $\pi_1$ to $\{1,\dots,n\}$ and $\pi_2$ to $\{n+1,\dots,n+m\}$. Then \cite{BWX} shows that for any pattern $\pi$, $S_n(12\dots k\oplus \pi)=S_n(k\dots21\oplus \pi)$. Combined, these results can be used to show that there are only three Wilf classes in $S_4$ \cite{Bona}.

For Latin Squares, we still have equivalence under reverse and complement. 
%  [CHECK THIS, perhaps thm 3.12 of $http://www.mat.unisi.it/newsito/puma/public_html/21_2/11_Phillipson_Riehl_Williams.pdf$].
\begin{thm}\label{comprev}
$L_n(\pi) = L_n(\pi^\text{rev}) = L_n(\pi^c)$ where $\pi^\text{rev}$ is the reverse of $\pi$ and $\pi^c$ is the complement of $\pi$.

\end{thm}

\begin{proof}
Let $\mathcal{L}_n(\pi)$ denote the set of $n^\text{th}$ order Latin Squares which avoid $\pi$. Given a Latin Square $S\in\mathcal{L}_n(\pi)$, let $\phi(S)$ be the new Latin Square when each entry $i$ is replaced with $n+1-i$. Since an occurrence of $\pi$ in $S$ would cause $\phi(S)$ to contain $\pi^c$, we then have that $\phi$ is a bijection from $\mathcal{L}_n(\pi)$ to $\mathcal{L}_n(\pi^c)$. To prove that $L_n(\pi)=L_n(\pi^\text{rev})$, we define the mapping $\rho$ by rotating the Latin Square $180^\circ$ around its center.
\begin{comment}
, as below:
\[
\begin{tabular}{ccc}
\begin{tabular}{|c|c|c|c|}
\hline
1&4&2&3\\\hline
4&3&1&2\\\hline
3&2&4&1\\\hline
2&1&3&4\\\hline
\end{tabular}
&
{\Large$\xrightarrow{\rho}$}
&
\begin{tabular}{|c|c|c|c|}
\hline
4&3&1&2\\\hline
1&4&2&3\\\hline
2&1&3&4\\\hline
3&2&4&1\\\hline
\end{tabular}
\end{tabular}
\]
\end{comment}
This has the effect of reversing all rows and columns, so this will be a bijection from $\mathcal{L}_n(\pi)$ to $\mathcal{L}_n(\pi^\text{rev})$.
\end{proof}

However, all of the other equivalences for patterns of length four do not carry over to Latin Squares. Dan Daly calculated $L_5(\pi)$ for every $\pi\in S_4$ (personal communication, July 12, 2012), using methods in \cite{McKay}, and the only equivalences that existed were the ones proved in Theorem \ref{comprev}. These  data show that, for avoidance in Latin Squares, there are eight Wilf classes in $S_4$ as opposed to the three in the case of permutations. This illustrates how pattern avoidance in Latin Squares is more nuanced and difficult than it is for permutations. Whether or not any nontrivial Wilf-equivalences exist for Latin Squares is an open question.

\section{Monotone Subsequences}
The celebrated theorem of Erd\H os and Szekeres states that every permutation of length $pq+1$ contains an increasing subsequence of length $p+1$ or a decreasing subsequence of length $q+1$. In the special case where $p=q$, we have that length $n^2+1$ permutations contain a $monotone$ (i.e. strictly increasing or strictly decreasing) subsequence of length $n+1$ \cite{Erdos}. In addition, this is the longest possible monotone sequence whose existence is guaranteed, since for every $m<n^2+1$, there exists permutations of length $m$ which have no monotone subsequence of length $n+1$. This result can be rephrased as follows:

\begin{thm}[Erd\H os and Szekeres]
Let $\lambda_n=\lfloor\sqrt{n-1}\rfloor+1$. Every permutation of length $n$ has a monotone subsequence of length $\lambda_n$, and there exist permutations of length $n$ without monotone subsequences of length $\lambda_n+1$. 
\end{thm}

In the above theorem, we can think of $\lambda_n$ as the length of the longest forced monotone subsequence. We wish to generalize this theorem to Latin Squares by defining a corresponding variable, $\Lambda_n$.
 
\begin{defn}
Let $\Lambda_n$ be the largest integer such that every  $n^{th}$ order Latin Square has a row or column with a monotone subsequence of length $\Lambda_n$. 
\end{defn}

It is trivially true that $\Lambda_n\ge\lambda_n$, since there is guaranteed to be a monotone subsequence of length $\lambda_n$ in every row and column of a given  $n^{th}$ order Latin Square. In the next theorem we present a slight improvement of this bound.

\begin{thm}
If $n\ge(m-1)(m-2)+2$ for some integer $m$, then $\Lambda_n\ge m$.
\end{thm}
\begin{proof}
Given an $n^{th}$ order Latin Square, consider the row whose leftmost entry is $n$. The $n-1$ entries to the right of this $n$ form a permutation of length at least $(m-1)(m-2)+1$. From the Erd\"os-Szekeres Result, we know that this permutation either has an increasing subsequence of length $m$ or a decreasing one of length $m-1$. In the latter case, combining the leftmost $n$ with this decreasing subsequence creates a decreasing subsequence of length $m$, so either way, a monotone sequence of length $m$ exists. Thus, $\Lambda_n\ge m$.
\end{proof}

Note that this proof could have been analogously applied to the column beginning with an $n$ or to the row and column beginning with a $1$, so that we are actually guaranteed four occurrences  of a monotone subsequence of length $m$. 

By inverting the formula in the condition of the previous theorem, we can express this result in terms of $\Lambda_n$. 
\begin{cor}\label{lb} For all $n>1$, 
$\Lambda_n\ge \left\lfloor \frac32 +\sqrt{n-\frac74}\right\rfloor.$
\end{cor}

We now show that this lower bound is tight for all perfect squares (except 1). The lower bound in these cases is given by
$$
\Lambda_{n^2}\ge\left\lfloor \frac32 +\sqrt{n^2-\frac74}\right\rfloor\ge\left\lfloor \frac32 +\sqrt{n^2-\frac{4n-1}4}\right\rfloor=n+1.
$$

We also have equality for these numbers, which was proved by Sam Connolly, a participant at the 2013 REU at East Tennessee State University. For $i,j\in \{1,2,\dots,n^2\}$, let $k_{ij}$ be in $\{1,2,\dots,n^2\}$ and satisfy $k_{ij}\equiv i+j-1$ (mod $n^2$). Consider the $n^2$-order Latin Square, where for $i,j\in \{1,2,\dots,n^2\}$, the $i,j$ entry is $k_{ij}n$ (mod $n^2+1$). For instance, when $n^2=9$, this produces the square in Figure \ref{connolly}. 

\begin{figure}
$$
\begin{array}{|c|c|c|c|c|c|c|c|c|}
\hline
3&6&9&2&5&8&1&4&7\\\hline
6&9&2&5&8&1&4&7&3\\\hline
9&2&5&8&1&4&7&3&6\\\hline
2&5&8&1&4&7&3&6&9\\\hline
5&8&1&4&7&3&6&9&2\\\hline
8&1&4&7&3&6&9&2&5\\\hline
1&4&7&3&6&9&2&5&8\\\hline
4&7&3&6&9&2&5&8&1\\\hline
7&3&6&9&2&5&8&1&4\\\hline
\end{array}
$$
\caption{Order 9 Latin Square whose longest monotone subsequence is 4.}\label{connolly}
\end{figure}

The first row and column is a classic example of a permutation of length $n^2$ whose longest monotone subsequence is of length $n$. The other rows and columns are cyclic permutations of the first, and this increases the length of this subsequence by only one. This example proves that $\Lambda_{n^2}\le n+1$. Combined with the lower bound, this shows $ \Lambda_{n^2}=n+1$.

Since one can show that $\Lambda_3=3$, our bound, that $\Lambda_3\ge2$, is not always tight. It would be true that $\Lambda_n$ was either equal to or one more than our lower bound if the conjecture below were true:

\begin{conj}
$\Lambda_{n}$ is nondecreasing. 
\end{conj}

The corresponding result is trivial for permutations, since a permutation of length $n+1$ naturally contains one of length $n$ by removing the $n+1$ entry. A similar containment argument does not translate immediately to  Latin Squares. 

\section{Open Questions}

We would like to end with a discussion of several open problems which we believe may spur future investigations.  Again, many of these problems were provided by attendees of the $11^{th}$ Permutation Patterns conference. We are unfortunately unable to remember all their names. Answering some of these questions would  lead to the beginning of a rich theory of Pattern Avoidance in Latin Squares.

\vspace{5mm}

\noindent{\bf Open Problems}	
\begin{itemize}

  \item What is $L_n(\pi)$ for patterns of length $4$ or more?  In particular, what can be said when $\pi = 1234$?
  
\item For a fixed pattern, say $\pi = 123...m$, can anything be said about the growth rate of $L_n(\pi)$? For which value of $m$ does the count first become exponential? 
\item Can anything be said about the growth rate of $L_n(\pi)$ vs $L_n(\pi')$, where $\pi$ and $\pi'$ are respectively patterns of length $i$ and $i+1$ and  $n\gg i$?
\item Which patterns are the easiest to avoid in Latin Squares?  The hardest?  If $\pi$ and $\pi'$ are of the same length, how different can $L_n (\pi)$ and $L_n(\pi ')$ be?  
\item Are there any Wilf-equivalences outside of those mentioned in Theorem \ref{comprev}?

\item What happens when pattern avoidance is defined to require the permutations induced by each symbol (see the footnote in Section 2) to avoid the target pattern as well? 

\item Can anything be said about Latin Squares avoiding a specific pattern (or set of patterns) in the rows, and a different pattern (or set of patterns) in the columns? For example, to avoid $123$ in the columns and $321$ in the rows, we can take a 123 avoiding square and reflect it through the vertical axis. This means that there are $n$ Latin Squares with this structure.  Can something more interesting be said using larger patterns?

\item Instead of avoidance or containment of specific patterns, can we build Latin Squares using a set of permutations in a (set of) avoidance classes?  How many can be built?

\item Is there a closed form expression for $L_n(\pi_n)$?  Such an expression, or even bounds, could, with Theorem 8, be used to find better bounds on $L_n.$

\end{itemize}

We end with a final unexplored generalization leading to additional open questions. In what follows, we define a \emph{Latin Rectangle} to be any rectangular array with entries in $1,\dots,n$ and no repeats in any row or column. The usual definition of a Latin Rectangle requires the number of columns to be $n$; in what follows, we wish to examine ``sub-rectangles'' induced by choosing any $p$ rows and $q$ columns of a Latin Square, and these subrectangles will not always fit the traditional definition. Call two Latin Rectangles \emph{order isomorphic} if one can be obtained from the other by applying an increasing function $f$ to each entry. For a Latin Rectangle, $R$, we say a Latin Square \emph{contains the pattern $R$} if it has some sub-rectangle which is order isomorphic to $R$. For example, consider the Latin Square in Figure \ref{connolly}. The subrectangle at rows $2$ and $7$, and columns $1,5$ and $9$, is shown below.
\[
\begin{array}{|c|c|c|}\hline
6&8&3\\\hline
1&6&8\\\hline
\end{array}
\]
This is order isomorphic to the rectangle,
\[R=
\begin{array}{|c|c|c|}\hline
3&4&2\\\hline
1&3&4\\\hline
\end{array}
\]
so we say the original square contains the pattern $R$.

Every problem addressed in the paper can be expressed in terms of rectangular patterns. For example, enumerating $L_n(123)$ is equivalent to counting Latin Squares which avoid both 
$R=\begin{array}{|c|c|c|}\hline
1&2&3\\\hline
\end{array}$
and the $90^\circ$ clockwise rotation of $R$. Many questions about permutation patterns generalize to rectangular patterns, so that there is a great deal of research that can be done in this area. 

\vspace{3mm}

\section{Acknowledgements}

The research of both authors was supported by NSF REU grant 1004624.  We thank Anant Godbole, Project Director, and the other participants for useful discussions.  We particularly thank Dan Daly for the data he provided, initially proposing the question of pattern avoidance in Latin Squares, and supporting us while we worked on the problem.  Many of the open questions were raised by participants at the 11th International Permutation Patterns Conference in Paris.  Finally, we thank our anonymous referees for their valuable feedback.

% Other questions

\begin{comment}
\item How many Latin Squares contain $123$?
\begin{itemize}
\item While this question is very straightforward to state, it may be hard to answer.  We end with it because any answer would allow you to readily calculate the total number of $n$-th order Latin Squares!
\end{itemize}
\end{itemize}

\item A Latin Rectangle is a $p$ by $q$ array with entries in $1,\dots,n$ and no repeats in any row or column. Can this problem be generalized to Latin Squares avoiding Latin Rectangular ``patterns?''
\end{comment}

%\begin{thebibliography}{99}
%\end{thebibliography}
\bibliographystyle{plain}
\bibliography{biblio}

\end{document}